\documentclass[12pt, reqno]{amsart}
\usepackage{amsmath, amsthm, amscd, amsfonts, amssymb, graphicx, color}
\usepackage[bookmarksnumbered, colorlinks, plainpages]{hyperref}

\newtheorem{theorem}{Theorem}[section]
\newtheorem{lemma}[theorem]{Lemma}

\newtheorem{corollary}[theorem]{Corollary}
\theoremstyle{definition}

\theoremstyle{remark}

\numberwithin{equation}{section}

\begin{document}

\setcounter{page}{1}

\title[Short Title]{On Some Fractional order Binomial sequence spaces with infinite Matrices}

\author[S.Dutta, S. Singh]{S.Dutta$^1$, S. Singh$^2$$^{*}$}

\address{$^{1}$ Department of Mathematics, Utkal University ,Odisha, India .}
\email{saliladutta516@gmail.com}

\address{$^{2}$ Department of Mathematics, Centurion University of Technology and Management, Odisha, India}
\email{ssaubhagyalaxmi@gmail.com, saubhagyalaxmi.singh@cutm.ac.in }

\subjclass[2010]{Primary 46A45; Secondary 46B45, 46A35.}

\keywords{Binomial difference sequence space, difference operator ,  $\left(\Delta ^{\tilde{\alpha}} \right)$, infinite matrix, Schauder basis, ${\alpha} -,\beta -$ and $\gamma -$duals.}

\date{Received: xxxxxx; Revised: yyyyyy; Accepted: zzzzzz.
\newline \indent $^{*}$ Corresponding author}

\begin{abstract}
The main purpose of this article is to introduce some new binomial difference sequence spaces of fractional order ${\tilde{\alpha}} $ along with infinite matrices. Some topological properties of these spaces are considered along with the Schauder basis and ${\alpha} -,\beta -$ and $\gamma -$duals of the spaces.
\end{abstract} \maketitle

\section{Introduction and preliminaries}
By  $\omega $,  we denote the space of all real valued sequences and any subspace of $\omega$ is called a sequence space. Let $c_{0} ,c $ and $l_{\infty }$  be the spaces of all  null, convergent and bounded sequences respectively which are normed by
 $\left\| x\right\| _{\infty }={\mathop{\sup }\limits_{k}} \left|x_{k} \right|$.
\noindent Again $l_{1} ,cs,bs$ denote the spaces of absolutely summable, convergent series and bounded series respectively. The space $l_{1} $ is normed by $\sum _{k}\left|x_{k} \right| $ and the spaces $cs,bs$ are normed by ${\mathop{\sup }\limits_{n}} \left|\sum _{k=0}^{n}x_{k}  \right|$, throughout this paper the summation without limit runs from 0 to $\infty $and $n\in {\mathbb{N}}^{+} =\left\{0,1,2,....\right\}$ .

The notion of difference sequence spaces was introduced by Kizmaz \cite{kizm15}  by considering  $X\left(\Delta \right)=\left\{x=\left(x_{k} \right):\Delta \left(x_{k} \right)\in X\right\}$ which is further generalized to m${}^{th}$  order difference sequence space as\\

\noindent $\Delta ^{m} \left(X\right)=\left\{x=\left(x_{k} \right):\Delta ^{m} (x)\in X\right\}$ for $X=\left\{c_{0}, c, l_{\infty }  \right\}$ \\

\noindent  by Et and \c{C}olak \cite{et13} where, $m$ be a non- negative integer.\\

\noindent Also $\Delta ^{m} \left(x\right)=\Delta ^{m-1} \left(x_{k} \right)-\Delta ^{m-1} \left(x_{k+1} \right)$

\noindent $\Delta ^{0} \left(x\right)=x_{k} $ and 
\[\Delta ^{m} \left(x_{k} \right)=\sum _{i=0}^{m}(-1)^{i}  \left(\begin{array}{l} {m} \\ {i} \end{array}\right)x_{k+i} \] 
The spaces $l_{\infty } ,c,c_{0} $,are Banach spaces normed by
\[\left\| x\right\| _{\Delta } =\sum _{i=1}^{n}\left|x_{i} \right| +{\mathop{\sup }\limits_{k}} \left|\Delta ^{m} \left(x_{k} \right)\right|.\] 
 Altay and Basar \cite{alta2} and Altay et al \cite{alta3} introduced the following Euler sequence spaces
\[\begin{array}{l} {e_{0}^{r} =\left\{x=\left(x_{k} \right)\in \omega :{\mathop{\lim }\limits_{n\to \infty }} \sum _{k}\left(\begin{array}{l} {n} \\ {k} \end{array}\right)(1-r)^{n-k} r^{k} x_{k} =0 \right\}} \\ 
{e_{c}^{r} =\left\{x=\left(x_{k} \right)\in \omega :{\mathop{\lim }\limits_{n\to \infty }} \sum _{k}\left(\begin{array}{l} {n} \\ {k} \end{array}\right)(1-r)^{n-k} r^{k} x_{k} \, exist \right\}} \\ 
 {e_{\infty }^{r} =\left\{x=\left(x_{k} \right)\in \omega :{\mathop{\sup }\limits_{n\in N}} \left|\sum _{k}\left(\begin{array}{l} {n} \\ {k} \end{array}\right)(1-r)^{n-k} r^{k} x_{k}  \right|<\infty \right\}} \end{array}\] 

where the $ r ^{th}$ order  Euler mean matrix $E^{r} $ is defined as$E^{r} =\left(e_{nk}^{r} \right)$,with $0<r<1$ and

\[e_{nk}^{r} =\left\{\begin{array}{l} {\left(\begin{array}{l} {n} \\ {k} \end{array}\right)(1-r)^{n-k} r^{k}, \qquad if\, 0\le k\le n} \\ {0,~\qquad \qquad \qquad \qquad \qquad  if\, k>n} \end{array}\right. \] 

Altay and Polat \cite{alta1} introduced and studied the Euler difference sequence spaces $Z\left(\Delta \right)$, for  $Z\in \left\{e_{0}^{r} ,e_{c}^{r} ,e_{\infty }^{r} \right\}$. Polat and Ba\c{s}ar \cite{pola17} further extended those space to\\
  $Z\left(\Delta ^{\left(m\right)} \right)=\left\{x=\left(x_{k} \right):\left(\Delta ^{\left(m\right)} \left(x_{k} \right)\right)\in Z, m\in \mathbb{N}\right\}$

\noindent where $Z\in \left\{e_{0}^{r} ,e_{c}^{r} ,e_{\infty }^{r} \right\}$and $\Delta ^{\left(m\right)} =\delta _{nk}^{\left(m\right)} $ is a triangle defined by  
\[\delta _{nk}^{\left(m\right)} =\left\{\begin{array}{l} {\left(-1\right)^{n-k} \left(\begin{array}{l} {m} \\ {n-k} \end{array}\right),\qquad if\, \max \{ 0,n-m\} \le k\le n} \\ {0,\qquad \qquad \qquad \qquad\qquad if\, 0\le k\le \max \left\{0,n-m\right\}\, or\, k>n} \end{array}\right. \]

\noindent The difference sequence spaces still attracted various mathematicians. In 2016 Bi\c{s}gin \cite{bisg8,bisg9} introduced the Binomial sequence spaces
\[\begin{array}{l} {b_{0}^{r,s} =\left\{x=\left(x_{k} \right)\in \omega :{\mathop{\lim }\limits_{n\to \infty }} \frac{1}{\left(s+r\right)^{n} } \sum _{k=0}^{n}\left(\begin{array}{l} {n} \\ {k} \end{array}\right)s^{n-k} r^{k} x_{k} =0 \right\}} \\ {b_{c}^{r,s} =\left\{x=\left(x_{k} \right)\in \omega :{\mathop{\lim }\limits_{n\to \infty }} \frac{1}{\left(s+r\right)^{n} } \sum _{k=0}^{n}\left(\begin{array}{l} {n} \\ {k} \end{array}\right)s^{n-k} r^{k} x_{k} \, exist \right\}} \\ {b_{\infty }^{r,s} =\left\{x=\left(x_{k} \right)\in \omega :{\mathop{\sup }\limits_{n\in N}} \left|\frac{1}{\left(s+r\right)^{n} } \sum _{k=0}^{n}\left(\begin{array}{l} {n} \\ {k} \end{array}\right)s^{n-k} r^{k} x_{k}  \right|<\infty \right\}} \end{array}\] 
with introducing the binomial matrix $B^{r,s} =\left(b_{nk}^{r,s} \right)$, defined by:
\[b_{nk}^{r,s} =\left\{\begin{array}{l} {\frac{1}{\left(s+r\right)^{n} } \left(\begin{array}{l} {n} \\ {k} \end{array}\right)s^{n-k} r^{k} \qquad \qquad if\, 0\le k\le n} \\ {0\qquad \qquad \qquad \qquad \qquad \qquad if\, k>n} \end{array}\right. \] 
where $r,s\in $ $\mathbb{R}$ and $r+s\ne 0$

\noindent For $r+s=1$, the Binomial matrix reduces to Euler matrix $E^{r} $.

\noindent Further  Meng and Song \cite{meng16} introduced the sequence space 
\[Z\left(\nabla ^{m} \right)=\left\{x=\left(x_{k} \right):\left(\nabla ^{\left(m\right)} x_{k} \right)\in Z\right\},\]  of order $m$ ,
for $Z\in \left\{b_{0}^{r,s} ,b_{c}^{r,s} ,b_{\infty }^{r,s} \right\}$and $\nabla ^{m} x_{k} =\sum _{i=0}^{m}\left(-1\right)^{i} \left(\begin{array}{l} {m} \\ {i} \end{array}\right)x_{k-i}  $. \\ 
 For a positive fraction $\tilde{\alpha}$, Baliarsingh \& Dutta (\cite{bali4,bali5,bali6,bali7,dutt10,dutt11}) they introduced the difference operator $\Delta ^{\tilde{\alpha}} $ 
\begin{equation} \label{1.1} 
\Delta ^{\tilde{\alpha} } x_{k} =\sum _{i}\left(-1\right)^{i} \frac{\Gamma \left({\tilde{\alpha} }+1\right)}{i!\Gamma ({\tilde{\alpha} }-i+1)} x_{k-i}   
\end{equation} 
with its inverse as 
\begin{equation} \label{1.2} 
\Delta ^{-\tilde{\alpha}} x_{k} =\sum _{i}\left(-1\right)^{i} \frac{\Gamma \left({-\tilde{\alpha} }+1\right)}{i!\Gamma ({-\tilde{\alpha} }-i+1)} x_{k-i}   
\end{equation} 
 $\Delta ^{\left({\tilde{\alpha} }\right)} $  can be expressed as a triangle
\[\left(\Delta ^{\tilde{\alpha} } \right)_{nk} =\left\{\begin{array}{l} {\left(-1\right)^{n-k} \frac{\Gamma \left({\tilde{\alpha} }+1\right)}{\left(n-k\right)!\Gamma ({\tilde{\alpha} }-n+k+1)} \qquad if\, 0\le k\le n} \\ {0\qquad \qquad \qquad \qquad \qquad if\, k>n} \end{array}\right. \] 
where $\Gamma \left(m\right)$ is a  Gamma function of all real numbers $m\notin \left\{0,-1,-2,....\right\}$, with
\begin{equation} \label{1.3} 
\Gamma \left(m\right)=\int _{0}^{\infty }e^{-x} x^{m-1} dx  
\end{equation} 
Now let $\lambda =\left(\lambda _{k} \right)_{k=0}^{\infty } $ be a strictly increasing sequence of positive reals tending to infinity, that is $0<\lambda _{0} <\lambda _{1} <....$ and $\lambda _{k} \to \infty $  as $\, k\to \infty $.
Mursaleen and Noman\cite{murs13} introduced the sequence spaces $l_{p}^{\lambda } $ and $l_{\infty }^{\lambda } $ of non- absolute type as the spaces of all sequences whose $\Lambda $- transforms are in the spaces $l_{p} $  and $l_{\infty } $ respectively.

where
\[\lambda _{nk} =\left\{\begin{array}{l} {\frac{\lambda _{k} -\lambda _{k-1} }{\lambda _{n} } \qquad if\, 0\le k\le n} \\ {0\qquad \qquad if\, k>n} \end{array}\right. \]  

\section{Main Results}

\noindent Now we define the product matrix  $\Lambda \left(B^{r,s} \left(\Delta ^{\tilde{\alpha }} \right)\right)$ and obtain their inverses and introduce binomial difference sequence spaces of fractional order
 $\tilde{\alpha}$, 
$\left[c_{0} \right]_{\Lambda \left(B^{r,s} \left(\Delta ^{\tilde{\alpha} } \right)\right)} ,\\
\left[c\right]_{\Lambda \left(B^{r,s} \left(\Delta ^{\tilde{\alpha} } \right)\right)} ,\left[l_{\infty } \right]_{\Lambda \left(B^{r,s} \left(\Delta ^{\tilde{\alpha} } \right)\right)} $and give some topological properties of the spaces. Combining the infinite matrix $\Lambda $ , binomial matrix $B^{r,s} $and the difference operator $\Delta ^{\tilde{\alpha}} $, the product matrix is  defined as
 
\begin{multline} \label{2.1} 
\left(\Lambda \left(B^{r,s} \left(\Delta ^{\tilde{\alpha} } \right)\right)\right)_{nk}\\
 =\left\{\begin{array}{l} {\sum _{i=k}^{n}\frac{1}{\lambda _{n} } \left(\lambda _{i} -\lambda _{i-1} \right)\frac{1}{\left(s+r\right)^{n} } \left(\begin{array}{l} {n} \\ {n-i} \end{array}\right)\frac{\Gamma \left(\tilde{\alpha} +1\right)}{\left(i-k\right)!\Gamma \left(\tilde{\alpha} -i+k+1\right)} r^{i} s^{n-i},  if\, 0\le k\le n } \\ {0,\qquad \qquad \qquad \qquad \qquad \qquad \qquad \qquad\qquad \qquad \qquad \qquad if\, k>n} \end{array}\right. 
\end{multline}

Equivalently

 \begin{multline}\label{*}
( \Lambda\left(B^{r,s}\left(\Delta^{\tilde{\alpha}}\right)\right)=\\
\left(%
\begin{array}{ccccccc}
  {\frac{\lambda _{0} -\lambda _{-1} }{\lambda _{0} } } & 0 & 0 & 0 & \dots \\
  ({\frac{\lambda _{0} -\lambda _{-1} }{\lambda _{0} } }) \frac{1}{(s+r)}(s-\tilde{\alpha}r ) &  ({\frac{\lambda _{1} -\lambda _{0} }{\lambda _{1} } }) \frac{1}{(s+r)}r  & 0 & 0 & \dots  \\
   ({\frac{\lambda _{0} -\lambda _{-1} }{\lambda _{2} } }) \frac{1}{(s+r)^2}(s^2-2\tilde{\alpha}sr+\frac{\tilde{\alpha}(\tilde{\alpha}-1)}{2!}r^2 ) &  ({\frac{\lambda _{1} -\lambda _{0} }{\lambda _{2} } }) \frac{1}{(s+r)^2}(2sr-\tilde{\alpha}r^2 ) & *  & 0&\dots  \\
   \vdots  & \vdots  & \vdots  &\vdots &\ddots
   \end{array}%
\right)
  \end{multline}
  
where *  means  $({\frac{\lambda _{2} -\lambda _{1} }{\lambda _{2} } }) \frac{1}{(s+r)^2}r^2$

Now we are interested  the inverse of above mentioned matrices. \\
Before proving certain theorems we now quote certain lemmas which will be used in Sequel.
\begin{lemma} \label{1}
\noindent The inverse of the difference matrix $\Delta ^{\tilde{\alpha} } $is given by the triangle
\[\left(\Delta ^{-\tilde{\alpha} } \right)_{nk} =\left\{\begin{array}{l} {\left(-1\right)^{n-k} \frac{\Gamma \left(-\tilde{\alpha} +1\right)}{\left(n-k\right)!\Gamma (-\tilde{\alpha} -n+k+1)} \qquad if\, 0\le k\le n} \\ {0\qquad \qquad \qquad \qquad  \qquad \qquad if\, k>n} \end{array}\right. \] 

\end{lemma}
\begin{lemma}\label{2}
\noindent The inverse of the binomial matrix $B^{r,s} $is given by the triangle
\[\left(B^{r,s} \right)_{nk}^{-1} =\left\{\begin{array}{l} {(-1)^{n-k} \left(s+r\right)^{k} \left(\begin{array}{l} {n} \\ {k} \end{array}\right)s^{n-k} r^{-n},  \qquad if\, 0\le k\le n} \\ {0,  \qquad  \qquad  \qquad \qquad \qquad \qquad \qquad if\, k>n} \end{array}\right. \] 
\end{lemma}
\begin{proof}
 The proof can be obtained easily using the method as provided in \cite{bali7,dutt10} and hence omitted.
 \end{proof}
\begin{lemma}\label{3}
\noindent The infinite matrix inverse
\[\Lambda ^{-1} =\lambda _{nk}^{-1} =\left\{\begin{array}{l} {(-1)^{n-k} \frac{\lambda _{n} }{\lambda _{k} -\lambda _{k-1} } \qquad if\, 0\le k\le n} \\ {0\qquad \qquad \qquad if\, k>n} \end{array}\right. \] 

\end{lemma}
\begin{theorem}\label{4}
\noindent The inverse of the product matrix  
$\left(\Lambda \left(B^{r,s} \left(\Delta ^{\tilde{\alpha} } \right)\right)\right)_{nk} $ is given by

\[\left(\Lambda \left(B^{r,s} \left(\Delta ^{\tilde{\alpha} } \right)\right)\right)_{nk}^{-1} =\\
\left\{\begin{array}{l} {\left(s+r\right)^{k} \sum _{i=k}^{n}\left(-1\right)^{n-k} \frac{\lambda _{n} }{\lambda _{k} -\lambda _{k-1} } \left(\begin{array}{l} {i} \\ {k} \end{array}\right)\frac{\Gamma \left(-\tilde{-\alpha} +1\right)}{\left(n-i\right)!\Gamma \left(\tilde{-\alpha}-n+i+1\right)} r^{-i} s^{i-k} ,  if\, 0\le k\le n } \\ {0, \qquad \qquad \qquad \qquad \qquad \qquad \qquad \qquad \qquad \qquad  \qquad  \qquad  \qquad  if\, k>n} \end{array}\right. \] 

\end{theorem}
\begin{proof}
 The result can be easily obtained by using lemma \ref{1}, lemma \ref{2} and lemma \ref{3}.
\end{proof}
\noindent 

\noindent Now we define the sequence spaces

$\left(\Lambda \left(B^{r,s} \left(\Delta ^{\tilde{\alpha} } \right)\right)\right)_{0} ,\Lambda \left(B^{r,s} \left(\Delta ^{\tilde{\alpha} } \right)\right),\left(\Lambda \left(B^{r,s} \left(\Delta ^{\tilde{\alpha} } \right)\right)\right)_{\infty } ,\tilde{l}\left(p\right)  as follows:$
\begin{tiny}
\[\begin{array}{l} {\left(\Lambda \left(B^{r,s} \left(\Delta ^{\tilde{\alpha} } \right)\right)\right)_{0}
=\left[c_{0} \right]_{\left(\Lambda \left(B^{r,s} \left(\Delta ^{\tilde{\alpha} } \right)\right)\right)} } \\
{=\left\{x=\left(x_{k} \right)\in w:{\mathop{\lim }\limits_{n\to \infty }} \sum _{k=0}^{n}\sum _{i=k}^{n}\frac{1}{\lambda _{n} } \left(\lambda _{k} -\lambda _{k-1} \right)\left(-1\right)^{i-k} \frac{1}{\left(s+r\right)^{n} } \left(\begin{array}{l} {n} \\ {n-i} \end{array}\right)
\frac{\Gamma \left(\tilde{\alpha} +1\right)}{\left(i-k\right)!\Gamma \left(\tilde{\alpha} -i+k+1\right)} r^{i} s^{n-i} x_{k}
 =0  \right\}} \end{array}\]

\[\begin{array}{l} {\left(\Lambda \left(B^{r,s} \left(\Delta ^{\tilde{\alpha} } \right)\right)\right)=\left[c\right]_{\left(\Lambda \left(B^{r,s} \left(\Delta ^{\tilde{\alpha} } \right)\right)\right)} } \\ {=\left\{x=\left(x_{k} \right)\in w:{\mathop{\lim }\limits_{n\to \infty }} \sum _{k=0}^{n}\sum _{i=k}^{n}\frac{1}{\lambda _{n} } \left(\lambda _{k} -\lambda _{k-1} \right)\left(-1\right)^{i-k} \frac{1}{\left(s+r\right)^{n} } \left(\begin{array}{l} {n} \\ {n-i} \end{array}\right)\frac{\Gamma \left(\tilde{\alpha} +1\right)}{\left(i-k\right)!\Gamma \left(\tilde{\alpha} -i+k+1\right)} r^{i} s^{n-i} x_{k} \, exist.  \right\}} \end{array}\] 
\[\begin{array}{l} {\left(\Lambda \left(B^{r,s} \left(\Delta ^{\tilde{\alpha} } \right)\right)\right)_{\infty } =\left[c_{0} \right]_{\left(\Lambda \left(B^{r,s} \left(\Delta ^{\tilde{\alpha} } \right)\right)\right)_{\infty } } } \\ {=\left\{x=\left(x_{k} \right)\in w:{\mathop{\sup }\limits_{n}} \left|\sum _{k=0}^{n}\sum _{i=k}^{n}\frac{1}{\lambda _{n} } \left(\lambda _{k} -\lambda _{k-1} \right)\left(-1\right)^{i-k} \frac{1}{\left(s+r\right)^{n} } \left(\begin{array}{l} {n} \\ {n-i} \end{array}\right)\frac{\Gamma \left(\tilde{\alpha} +1\right)}{\left(i-k\right)!\Gamma \left(\tilde{\alpha} -i+k+1\right)} r^{i} s^{n-i} x_{k}   \right|<\infty \right\}} \end{array}\] 
\[\begin{array}{l} {\tilde{l}\left(p\right)=\left[l\left(p\right)\right]_{\left(\Lambda \left(B^{r,s} \left(\Delta ^{\tilde{\alpha} } \right)\right)\right)} } \\ {=\left\{x=\left(x_{k} \right)\in w:\sum _{n}\left|\sum _{k=0}^{n}\sum _{i=k}^{n}\frac{1}{\lambda _{n} } \left(\lambda _{k} -\lambda _{k-1} \right)\left(-1\right)^{i-k} \frac{1}{\left(s+r\right)^{n} } \left(\begin{array}{l} {n} \\ {n-i} \end{array}\right)\frac{\Gamma \left(\tilde{\alpha} +1\right)}{\left(i-k\right)!\Gamma \left(\tilde{\alpha} -i+k+1\right)} r^{i} s^{n-i} x_{k}   \right|^{p} <\infty  \right\}} \end{array}\] 

\end{tiny}
For different values of $\tilde{\alpha}, s, r $ and for $\Lambda=1$, of spaces defined above generalize the spaces defined in \cite{alta2, pola17,kada14,bisg8,meng16,yayi19}

\noindent Define the sequence $y=\left(y_{k} \right)$ by the \\
 $\left(\Lambda \left(B^{r,s} \left(\Delta ^{\tilde{\alpha} } \right)\right)\right)-$transform of a sequence $x=\left(x_{k} \right)$, i.e.
\begin{equation} \label{2.2} 
\begin{array}{l} {y_{k} =\left(\Lambda \left(B^{r,s} \left(\Delta ^{\tilde{\alpha} } \right)x\right)\right)_{k} } \\ {=\sum _{j=0}^{k}\sum _{i=j}^{k}\frac{1}{\lambda _{k} } \left(\lambda _{j} -\lambda _{j-1} \right)\left(-1\right)^{i-j} \frac{1}{\left(s+r\right)^{k} } \left(\begin{array}{l} {k} \\ {k-i} \end{array}\right)\frac{\Gamma \left(\tilde{\alpha} +1\right)}{\left(i-j\right)!\Gamma \left(\tilde{\alpha} -i+j+1\right)} r^{i} s^{k-i} x_{j}   } \end{array} 
\end{equation} 
\begin{theorem}\label{th2.5}
\noindent The sequence space $X\left(\Delta ^{\left(\tilde{\alpha} \right)} \right)$ is a BK-Space with the norm defined by
\[\begin{array}{l} {\left\| x\right\| _{X\left(\Delta ^{\left(\tilde{\alpha} \right)} \right)} =\left\| y\right\| _{\infty } } \\ {\qquad =\left\| \left(\Lambda \left(B^{r,s} \left(\Delta ^{\tilde{\alpha} } \right)\right)x\right)_{k} \right\| _{\infty } } \end{array}\] 
where $X\in \left\{\left(\Lambda \left(B^{r,s} \left(\Delta ^{\tilde{\alpha} } \right)\right)\right)_{0} ,\Lambda \left(B^{r,s} \left(\Delta ^{\tilde{\alpha} } \right)\right),\left(\Lambda \left(B^{r,s} \left(\Delta ^{\tilde{\alpha}} \right)\right)\right)_{\infty } \right\}$.

\end{theorem}
\begin{proof}
The proof is a routine verification and hence omitted.
\end{proof}

\begin{theorem}\label{th2.6}
\noindent The sequence spaces $\left(\Lambda \left(B^{r,s} \left(\Delta ^{\tilde{\alpha} } \right)\right)\right)_{0} ,\Lambda \left(B^{r,s} \left(\Delta ^{\tilde{\alpha}} \right)\right),\left(\Lambda \left(B^{r,s} \left(\Delta ^{\tilde{\alpha} } \right)\right)\right)_{\infty } $ \\
are linearly isomorphic to $c_{0} ,\, c,\, l_{\infty } $ respectively.

\end{theorem}
\begin{proof}
 The result will be proved for the space $\left(\Lambda \left(B^{r,s} \left(\Delta ^{\tilde{\alpha} } \right)\right)\right)_{0} $.

\noindent For other spaces the results can follow in a similar manner.

\noindent Now define a mapping $T:\left(\Lambda \left(B^{r,s} \left(\Delta ^{\tilde{\alpha} } \right)\right)\right)_{0} \to c_{0} $ by      $x\to y=Tx$.

\noindent Clearly whenever  $Tx=0$, T is linear. \\
Which implies $T$ is injective.

\noindent Let $y\in c_{0} $  we define a sequence $x=\left(x_{k} \right)$ by
\begin{equation} \label{2.3} 
x_{k} =\sum _{i=0}^{k}(s+r)^{i} \sum _{j=i}^{k}\left(-1\right)^{k-i} \frac{\lambda _{k} }{\lambda _{i} -\lambda _{i-1} } \left(\begin{array}{l} {j} \\ {i} \end{array}\right)\frac{\Gamma \left(-\tilde{\alpha} +1\right)}{\left(k-j\right)!\Gamma \left(-\tilde{\alpha} -k+j+1\right)} r^{-j} s^{j-i} y_{i}    
\end{equation} 
using theorem \ref{4},
Then we have
\[\begin{array}{l} {{\mathop{\lim }\limits_{n\to \infty }} \left(\Lambda \left(B^{r,s} \left(\Delta ^{\tilde{\alpha} } \right)x\right)\right)_{n} } \\ {={\mathop{\lim }\limits_{n\to \infty }} \sum _{j=0}^{n}\sum _{i=j}^{n}\left(-1\right)^{i-j} \frac{1}{\lambda _{n} } \left(\lambda _{j} -\lambda _{j-1} \right)\frac{1}{\left(s+r\right)^{n} } \left(\begin{array}{l} {n} \\ {n-i} \end{array}\right)\frac{\Gamma \left(\tilde{\alpha} +1\right)}{\left(i-j\right)!\Gamma \left(\tilde{\alpha} -i+j+1\right)} r^{i} s^{k-i} x_{j}   } \\ {={\mathop{\lim }\limits_{n\to \infty }} y_{n} } \\ {=0} \end{array}\] 
Therefore 
\[x\in \left(\Lambda \left(B^{r,s} \left(\Delta ^{\tilde{\alpha} } \right)\right)\right)_{0} \] 
and  $y=Tx$

\noindent Which implies $T$ is surjective and norm preserving.

\noindent i.e.  $\left(\Lambda \left(B^{r,s} \left(\Delta ^{\tilde{\alpha} } \right)\right)\right)_{0} \cong c_{0} $.
\end{proof} 
\section{Schauder Basis}

\noindent This section deals with Schauder basis for the sequence spaces $\left(\Lambda \left(B^{r,s} \left(\Delta ^{\tilde{\alpha} } \right)\right)\right)_{0} $\\
and $\left(\Lambda \left(B^{r,s} \left(\Delta ^{\tilde{\alpha} } \right)\right)\right)$. A sequence $\left(x_{k} \right)$of a normed space $\left(X,\, \left\| \, .\, \right\| \right)$ is called a Schauder basis if for every $u\in X$there exist an unique sequence of scalars $\left(a_{k} \right)$ such that 
\[{\mathop{\lim }\limits_{n\to \infty }} \left\| u-\sum _{k=0}^{n}a_{k} x_{k}  \right\| =0\] 
Define the sequence ${}^{(k)} \theta ^{r,s} =\left({}^{(k)} \theta _{{}^{n} }^{r,s} \right)_{n\in\mathbb{N} } $by

\noindent ${}^{(k)} \theta _{n}^{r,s} =\left\{\begin{array}{l} {\left(s+r\right)^{k} \sum _{i=k}^{n}\left(-1\right)^{n-k} \frac{\lambda _{k} -\lambda _{k-1} }{\lambda _{k} } \left(\begin{array}{l} {i} \\ {k} \end{array}\right)\frac{\Gamma \left(-\tilde{\alpha} +1\right)}{\left(n-i\right)!\Gamma \left(-\tilde{\alpha} -n+i+1\right)} r^{-i} s^{i-k} \qquad  if\, 0\le k\le n } \\ {0\qquad \qquad \qquad \qquad \qquad \qquad \qquad \qquad \qquad \qquad \, \, \, \, \, \, \, if\, k>n} \end{array}\right. $
for each $k\in \mathbb{N}$ .
\begin{theorem}\label{th3.1}
\noindent The sequence $\left({}^{(k)} \theta ^{r,s} \right)$ is a Schauder basis for the sequence space $\left(\Lambda \left(B^{r,s} \left(\Delta ^{\tilde{\alpha} } \right)\right)\right)_{0} $and every $x\in \left(\Lambda \left(B^{r,s} \left(\Delta ^{\tilde{\alpha} } \right)\right)\right)_{0} $has unique representation of the form
\[x=\sum _{k}\sigma _{k}^{r,s} {}^{\left(k\right)} \theta ^{r,s}  \] 
where $\sigma _{k}^{r,s} =\left[\Lambda \left(B^{r,s} \left(\Delta ^{\tilde{\alpha} } \right)\right)x\right]_{k} $for each $k\in \mathbb{N}$.

\end{theorem}
\begin{proof}
\noindent  Using the definition of $\left(\Lambda \left(B^{r,s} \left(\Delta ^{\tilde{\alpha} } \right)\right)\right)$ and $\left({}^{(k)} \theta ^{r,s} \right)$, we can easily verify that 
\[\left(\left(\Lambda \left(B^{r,s} \left(\Delta ^{\tilde{\alpha} } \right)\right)\right){}^{(k)} \theta ^{r,s} \right)=e^{\left(k\right)} \in c_{0} ,\] 

where $e^{\left(k\right)}$ is a sequence with 1  in the $kth$ place and zeros elsewhere. So, the inclusion $^{(k)} \theta ^{r,s}  \in \left[\Lambda \left(B^{r,s} \left(\Delta ^{\tilde{\alpha}} \right)\right)x\right]_{0}$.

\noindent The set $\left\{e^{\left(k\right)} :k\in \mathbb{N}\right\}$is the Schauder basis for the space $c_{0} $. Because the isomorphism T defined by $x\to y=Tx$ by (see Theorem \ref{th2.6}) from the space $\left(\Lambda \left(B^{r,s} \left(\Delta ^{\tilde{\alpha} } \right)\right)\right)_{0} $to $c_{0} $is onto, therefore the inverse image of the basis of space $c_{0} $ forms the basis of $\left(\Lambda \left(B^{r,s} \left(\Delta ^{\tilde{\alpha} } \right)\right)\right)_{0} $, i.e. 
\[{\mathop{\lim }\limits_{n\to \infty }} \left\| x-\sum _{k=0}^{n}\sigma _{k}^{r,s} {}^{\left(k\right)} \theta ^{r,s}  \right\| =0,x\in \left(\Lambda \left(B^{r,s} \left(\Delta ^{\tilde{\alpha} } \right)\right)\right)_{0} .\] 
To verify the uniqueness of the representation assume that $x=\sum _{k}\mu _{k}^{r,s} {}^{\left(k\right)} \theta ^{r,s}  $then we have
\[\left[\Lambda \left(B^{r,s} \left(\Delta ^{\tilde{\alpha} } \right)\right)x\right]_{k} =\sum _{k}\mu _{k}^{r,s}  \left(\Lambda \left(B^{r,s} \left(\Delta ^{\tilde{\alpha} } \right)\right){}^{(k)} \theta {}^{r,s} \right)_{k} =\sum _{k}\mu _{k}^{r,s} e^{\left(k\right)}  =\mu _{k}^{r,s} .\] 
This contradict to our assumption that $\mu _{k}^{r,s} =\left[\Lambda \left(B^{r,s} \left(\Delta ^{\tilde{\alpha} } \right)\right)x\right]_{k} $ for each \\
$k\in \mathbb{N}$. 
\end{proof}

\begin{theorem}\label{th3.2}
\noindent Define $\eta =\eta _{k} $by
\[\eta _{k} =\sum _{i=0}^{k}\left(s+r\right)^{i} \sum _{j=i}^{k}\left(-1\right)^{n-i}   \left(\begin{array}{l} {j} \\ {i} \end{array}\right)\frac{\Gamma \left(-\tilde{\alpha} +1\right)}{(n-j)!\Gamma \left(-\tilde{\alpha} -n+j+1\right)} r^{-j} s^{j-i} ,\,\,  k,n\in\mathbb{N}\] 
and ${\mathop{\lim }\limits_{k\to \infty }} \sigma _{k}^{r,s} =l$. Then the set $\left\{\eta ,{}^{(0)} \theta ^{r,s} ,{}^{(1)} \theta ^{r,s} ,....\right\}$is a Schauder basis for the space $\left(\Lambda \left(B^{r,s} \left(\Delta ^{\tilde{\alpha} } \right)\right)\right)_{c} $and every $x\in \left(\Lambda \left(B^{r,s} \left(\Delta ^{\tilde{\alpha} } \right)\right)\right)_{0} $has a unique representation of the form
\[x=l\eta +\sum _{k}\left(\sigma _{k}^{r,s} -l\right) ^{\left(k\right)} \theta ^{r,s} .\] 

\end{theorem}
\begin{proof}
The proof as it is similar to previous theorem.
\end{proof}
 
\begin{corollary}\label{cr3.3}
\noindent The sequence spaces $\left(\Lambda \left(B^{r,s} \left(\Delta ^{\tilde{\alpha} } \right)\right)\right)_{0} $and $\left(\Lambda \left(B^{r,s} \left(\Delta ^{\tilde{\alpha} } \right)\right)\right)_{c} $are separable.

\end{corollary}
\begin{proof}

 The result follows from the theorems \ref{th2.5},  \ref{th3.1}, \ref{th3.2}. 
 \end{proof}
\section{${\alpha}, \beta$ and $\gamma$ duals}

This section deals with $\alpha -,\beta -$ and$\gamma -$duals  of  $\left(\Lambda _{0} \left(B^{r,s} \left(\Delta ^{\left(\tilde{\alpha} \right)} \right)\right)\right)$ and\\
$\left(\Lambda _{c} \left(B^{r,s} \left(\Delta ^{\left(\tilde{\alpha} \right)} \right)\right)\right)$. For the sequence spaces X and Y, define multiplier sequence space $M(X,Y)$by 
\[M(X,Y)=\left\{u=\left(u\right)_{k} \in \omega :ux=\left(u_{k} x_{k} \right)\in Y,whenever\, x=\left(x_{k} \right)\in X\right\}\] 
Let $\alpha -,\beta -$ and$\gamma -$duals be denoted by 

\noindent $X^{\alpha } =M\left(X,l_{1} \right),\, X^{\beta } =M\left(X,cs\right),\, X^{\gamma } =M\left(X,bs\right)$ respectively.

\noindent Throughout $\tau $ will denote the collection of all finite subsets of $ \mathbb{N}$. We consider $K\in \tau$.

\noindent We now quote the following results which will be used for finding the duals.
\begin{equation} \label{4.1} 
{\mathop{\sup }\limits_{k\in \tau }} \sum _{n}\left|\sum _{k\in K}a_{n,k}  \right| <\infty  
\end{equation} 
\begin{equation} \label{4.2} 
{\mathop{\sup }\limits_{n\in N}} \sum _{k}\left|a_{n,k} \right| <\infty  
\end{equation} 
\begin{equation} \label{4.3} 
{\mathop{\lim }\limits_{n\to \infty }} a_{n,k} =a_{k} ,\qquad for\, each\, k\in N 
\end{equation} 
\begin{equation} \label{4.4} 
{\mathop{\lim }\limits_{n\to \infty }} \sum _{k}a_{n,k}  =a 
\end{equation} 
\begin{equation} \label{4.5} 
{\mathop{\lim }\limits_{n\to \infty }} \sum _{k}\left|a_{n,k} \right| =\sum _{k}\left|{\mathop{\lim }\limits_{n\to \infty }} a_{n,k} \right|  
\end{equation} 
\begin{lemma}\cite{stie18}\label{lm4}
\noindent Let $A=\left(a_{n,k} \right)$be an infinite matrix, then

\begin{enumerate}
\item  $A\in \left(c_{o} :l_{1} \right)=\left(c:l_{1} \right)=\left(l_{\infty } :l_{1} \right)$ iff \eqref{4.1} holds.

\item  $A\in \left(c_{o} :c\right)$ iff \eqref{4.2} ,\eqref{4.3} hold.

\item  $A\in \left(c:c\right)$ iff \eqref{4.2}, \eqref{4.3}, \eqref{4.4} hold.

\item  $A\in \left(l_{\infty } :c\right)$ iff \eqref{4.3}  and \eqref{4.5}hold.

\item  $A\in \left(c_{o} :l_{\infty } \right)=\left(c:l_{\infty } \right)=\left(l_{\infty } :l_{\infty } \right)$ iff \eqref{4.2} holds.
\end{enumerate}

\end{lemma} 
\begin{theorem}
\noindent The $\alpha -$duals of  the spaces $\left(\Lambda _{0} \left(B^{r,s} \left(\Delta ^{\left(\tilde{\alpha} \right)} \right)\right)\right)$,$\left(\Lambda _{c} \left(B^{r,s} \left(\Delta ^{\left(\tilde{\alpha} \right)} \right)\right)\right)$ and $\left(\Lambda _{\infty } \left(B^{r,s} \left(\Delta ^{\left(\tilde{\alpha} \right)} \right)\right)\right)$ is the set 
\begin{tiny}
\[D_{1}^{r,s} =\\
\left\{d=\left(d_{k} \right)\in \omega :{\mathop{\sup }\limits_{k\in \tau }} \sum _{k}\left[\left|\sum _{i\in k}(s+r)^{k} (-1)^{n-k} \frac{\lambda _{k} }{\lambda _{i} -\lambda _{i-1} }  \sum _{j=i}^{k}\left(\begin{array}{l} {j} \\ {k} \end{array}\right)\frac{\Gamma \left(-\tilde{\alpha} +1\right)}{\left(n-j\right)!\\
\Gamma \left(-\tilde{\alpha} -n+j+1\right)} s^{j-k} r^{-j}  \right|\right]\left|d_{k} \right|<\infty  \right\}\] 
\[\left(\Lambda _{0} \left(B^{r,s} \left(\Delta ^{\left(\tilde{\alpha} \right)} \right)\right)\right)=\left(\Lambda _{c} \left(B^{r,s} \left(\Delta ^{\left(\tilde{\alpha} \right)} \right)\right)\right)^{\tilde{\alpha} } =\left(\Lambda _{\infty } \left(B^{r,s} \left(\Delta ^{\left(\tilde{\alpha} \right)} \right)\right)\right)=D_{1}^{r,s} \]
\end{tiny}

\end{theorem}
\begin{proof}

\noindent Considering $x=\left(x_{k} \right)$ as in \ref{2.3}, let $d=\left(d_{k} \right)\in \omega $ define
\[\begin{array}{l} {d_{n} x_{n} =\sum _{i=0}^{n}(s+r)^{i} (-1)^{n-i} \frac{\lambda _{k} }{\lambda _{i} -\lambda _{i-1} } \sum _{j=i}^{n}\left(\begin{array}{l} {j} \\ {i} \end{array}\right)\frac{\Gamma \left(-\tilde{\alpha} +1\right)}{\left(n-j\right)!\Gamma \left(-\tilde{\alpha} -n+j+1\right)} r^{-j} s^{j-i} d_{n} y_{i}   } \\ {\, \, \, \, \, \, \, \, \, \, =\left(D^{r,s} y\right)_{n} ,\qquad {\rm for}\, {\rm each}\, n\in N} \end{array}\] 
where $D^{r,s} =\left(d_{n,k}^{r,s} \right)$ is defined by
\[d_{n,k}^{r,s} =\left\{\begin{array}{l} {(s+r)^{k} \frac{\lambda _{k} }{\lambda _{i} -\lambda _{i-1} } (-1)^{n-k} \sum _{j=i}^{n}\left(\begin{array}{l} {j} \\ {k} \end{array}\right)\frac{\Gamma \left(-\tilde{\alpha} +1\right)}{\left(n-j\right)!\Gamma \left(-\tilde{\alpha} -n+j+1\right)} s^{j-k} r^{-j} d_{n}, \qquad if\, 0\le k\le n } \\ {0, \qquad \qquad \qquad \qquad \qquad \qquad \qquad \qquad \qquad  \qquad \qquad \qquad  \qquad \qquad if\, k>n} \end{array}\right. \] 
Therefore we deduce that

\noindent $dx=\left(d_{n} x_{n} \right)\in l_{1} $ whenever  $x\in \left(\Lambda _{0} \left(B^{r,s} \left(\Delta ^{\left(\tilde{\alpha}\right)} \right)\right)\right)$ or $x\in \left(\Lambda _{c} \left(B^{r,s} \left(\Delta ^{\left(\tilde{\alpha} \right)} \right)\right)\right)$ or $x\in \left(\Lambda _{\infty } \left(B^{r,s} \left(\Delta ^{\left(\tilde{\alpha} \right)} \right)\right)\right)$ if and only if $D^{r,s} y\in l_{1} $ whenever $y\in c_{0} ,c\, \& \, l_{\infty } $, which implies that $d=\left(d_{n} \right)\in \left[\Lambda _{0} \left(B^{r,s} \left(\Delta ^{\left(\tilde{\alpha} \right)} \right)\right)\right]^{\tilde{\alpha} } =\left[\Lambda _{c} \left(B^{r,s} \left(\Delta ^{\left(\tilde{\alpha} \right)} \right)\right)\right]^{\tilde{\alpha} } =\left[\Lambda _{\infty } \left(B^{r,s} \left(\Delta ^{\left(\tilde{\alpha} \right)} \right)\right)\right]^{\tilde{\alpha} } $

\noindent if and only if $D^{r,s} \in \left(c_{0} :l_{1} \right)=\left(c:l_{1} \right)=\left(l_{\infty } :l_{1} \right)$ by lemma \ref{lm4}(i), we obtain

\noindent iff 
\[{\mathop{\sup }\limits_{k\in \tau }} \sum _{k}\left|\sum _{i\in K}(s+r)^{k} (-1)^{n-k} \frac{\lambda _{k} }{\lambda _{i} -\lambda _{i-1} } \sum _{j=k}^{n}\left(\begin{array}{l} {j} \\ {k} \end{array}\right)\frac{\Gamma \left(-\tilde{\alpha} +1\right)}{\left(n-j\right)!\Gamma \left(-\tilde{\alpha} -n+j+1\right)} r^{-j} s^{j-k} d_{k}   \right| <\infty \] 
Thus we have 
\[\left(\Lambda _{0} \left(B^{r,s} \left(\Delta ^{\left(\tilde{\alpha} \right)} \right)\right)\right)=\left(\Lambda _{c} \left(B^{r,s} \left(\Delta ^{\left(\tilde{\alpha} \right)} \right)\right)\right)^{\tilde{\alpha} } =\left(\Lambda _{\infty } \left(B^{r,s} \left(\Delta ^{\left(\tilde{\alpha} \right)} \right)\right)\right)=D_{1}^{r,s} \] 
\end{proof}
\begin{theorem}
\noindent Now we define the sets $D_{2}^{r,s} ,D_{3}^{r,s} $and $D_{4}^{r,s} $by 
\[\begin{array}{l} {D_{2}^{r,s} =\left\{d=\left(d_{k} \right)\in \omega :{\mathop{\sup }\limits_{n\in N}} \sum _{K}\left|t_{n,k}^{r,s} \right|<\infty  \right\}} \\ {D_{3}^{r,s} =\left\{d=\left(d_{k} \right)\in \omega :{\mathop{\lim }\limits_{n\to \infty }} t_{n,k}^{r,s} \, {\rm exists}\, {\rm for}\, {\rm all}\, k\in N\right\}} \\ {D_{4}^{r,s} =\left\{d=\left(d_{k} \right)\in \omega :{\mathop{\lim }\limits_{n\to \infty }} \sum _{K}t_{n,k}^{r,s}  \, {\rm exists}\right\}} \end{array}\] 
where 
\[T^{r,s} =t_{n,k}^{r,s} =\left\{\begin{array}{l} {\sum _{i=k}^{n}(s+r)^{k} (-1)^{n-k} \frac{\lambda _{k} }{\lambda _{i} -\lambda _{i-1} } \sum _{j=k}^{n}\left(\begin{array}{l} {j} \\ {k} \end{array}\right)\frac{\Gamma \left(-\tilde{\alpha} +1\right)}{\left(i-j\right)!\Gamma \left(-\tilde{\alpha} -i+j+1\right)} r^{-j} s^{j-k} z_{j} ,\,  if\, 0\le k\le n  } \\ {0\qquad if\, k>n} \end{array}\right. \] 
Then 
\[\begin{array}{l} {(i)\left(\Lambda _{0} \left(B^{r,s} \left(\Delta ^{\left(\tilde{\alpha} \right)} \right)\right)\right)^{\beta } =D_{2}^{r,s} \cap D_{3}^{r,s} } \\ {(ii)\left(\Lambda _{c} \left(B^{r,s} \left(\Delta ^{\left(\tilde{\alpha} \right)} \right)\right)\right)^{\beta } =D_{2}^{r,s} \cap D_{3}^{r,s} \cap D_{4}^{r,s} } \\ {(iii)\left(\Lambda _{\infty } \left(B^{r,s} \left(\Delta ^{\left(\tilde{\alpha} \right)} \right)\right)\right)^{\beta } =D_{3}^{r,s} \cap D_{4}^{r,s} } \end{array}\] 

\end{theorem}

\begin{proof}

\noindent$ (i)$  Let $d=\left(d_{k} \right)\in \omega $ and $x=\left(x_{k} \right)$ is defined as in \eqref{4}. Then 
\[\begin{array}{l} {\sum _{k=0}^{n}d_{k} x_{k}  =\sum _{k=0}^{n}d_{k}  \sum _{i=0}^{k}(s+r)^{i} (-1)^{k-i} \frac{\lambda _{k} }{\lambda _{i} -\lambda _{i-1} } \sum _{j=i}^{k}\left(\begin{array}{l} {j} \\ {i} \end{array}\right)\frac{\Gamma \left(-\tilde{\alpha} +1\right)}{\left(k-j\right)!\Gamma \left(-\tilde{\alpha} -k+j+1\right)} r^{-j} s^{j-i} y_{i}   } \\ {\qquad =\sum _{k=0}^{n}\left[\sum _{i=k}^{n}(s+r)^{k} (-1)^{i-k} \frac{\lambda _{i} }{\lambda _{k} -\lambda _{k-1} } \sum _{j=k}^{i}\left(\begin{array}{l} {j} \\ {k} \end{array}\right)\frac{\Gamma \left(-\tilde{\alpha} +1\right)}{\left(i-j\right)!\Gamma \left(-\tilde{\alpha} -i+j+1\right)} r^{-j} s^{j-k} d_{i}   \right] y_{k} } \\ {\qquad =\left(T^{r,s} y\right)_{n} ,\qquad {\rm for}\, {\rm each}\, n\in N} \end{array}\] 
Hence $dx=\left(d_{k} x_{k} \right)\in cs$whenever $x\in \left(\Lambda _{0} \left(B^{r,s} \left(\Delta ^{\left(\tilde{\alpha} \right)} \right)\right)\right)$if and only if $T^{r,s} y\in c$whenever $y\in c_{0} $, which implies that $d=\left(d_{k} \right)\in \left[\Lambda _{0} \left(B^{r,s} \left(\Delta ^{\left(\tilde{\alpha}\right)} \right)\right)\right]^{\beta } $if and only if $T^{r,s} \in \left(c_{0} :c\right)$.

\noindent By lemma \ref{4.2}, we obtain
\[\left[\Lambda _{0} \left(B^{r,s} \left(\Delta ^{\left(\tilde{\alpha} \right)} \right)\right)\right]^{\beta } =D_{2}^{r,s} \cap D_{3}^{r,s} .\] 
Then proves for $(ii) $ and $(iii)$ can be obtained in similar manner.
\end{proof}
\begin{theorem}
\noindent The $\gamma -$dual of the spaces $\left(\Lambda _{0} \left(B^{r,s} \left(\Delta ^{\left(\tilde{\alpha}\right)} \right)\right)\right),\, \left(\Lambda _{c} \left(B^{r,s} \left(\Delta ^{\left(\tilde{\alpha}\right)} \right)\right)\right)^{\tilde{\alpha} } \,  and\\
\left(\Lambda _{\infty } \left(B^{r,s} \left(\Delta ^{\left(\tilde{\alpha} \right)} \right)\right)\right)$ in the set $D_{2}^{r,s} $.
\end{theorem}
\begin{proof}

\noindent  As it is a routine verification we omit the proof.
\end{proof}

 Now using Geometric sequence spaces formulas in  $ \left(\Lambda \left(B^{r,s} \left(\Delta ^{\tilde{\alpha} } \right)\right)\right) $ its written as $ \left(\Lambda \left(B^{r,s} \left(\Delta ^{\tilde{\alpha} } \right)\right)\right)_G $  as
 \begin{equation}\label{*}
( \Lambda\left(B^{r,s}\left(\Delta^{\tilde{\alpha}}\right)\right)_G=\\
\left(%
\begin{array}{ccccccc}
  {\frac{e} {\lambda _{-1}^G } } & 1 & 1 & 1 & \dots \\
  {\frac{e }{\tilde{\alpha}\lambda _{-1}^G (e^2\odot r) } }  &  {\frac{e }{s\lambda _{0}^G  }}  & 1 & 1 & \dots  \\
   ({\frac{\lambda _{0}}{\lambda _{-1}\odot\lambda_2 } }) \frac{e}{2sr\odot2!_G} &  {\frac{\lambda _{1} }{\lambda _{0}\lambda_{2 } }} {\frac{2}{\tilde{\alpha}sr\odot e^3}} & {\frac{e }{\lambda _{1}e^2\odot s }}  & 1&\dots  \\
   \vdots  & \vdots  & \vdots  &\vdots &\ddots
   \end{array}%
\right)
  \end{equation}

\bibliographystyle{amsplain}

\end{document}